\documentclass[11pt,leqno]{article}
\usepackage[doc]{optional}
\usepackage{color}
\usepackage{soul}
\usepackage{graphicx}
\definecolor{labelkey}{rgb}{0,0.08,0.45}
\definecolor{refkey}{rgb}{0,0.6,0.0}
\definecolor{Brown}{rgb}{0.45,0.0,0.05}
\definecolor{lime}{rgb}{0.00,0.8,0.0}
\definecolor{lblue}{rgb}{0.5,0.5,0.99}
\usepackage{mathpazo}
\usepackage{amsmath}
\usepackage{amssymb}
\usepackage{theorem}
\newcommand{\reckti}{rectangular}

\newcommand{\nnn}{\ensuremath{{n\in{\mathbb N}}}}
\newcommand{\thalb}{\ensuremath{\tfrac{1}{2}}}
\newcommand{\menge}[2]{\big\{{#1}~\big |~{#2}\big\}}

\newcommand{\To}{\ensuremath{\rightrightarrows}}

\newcommand{\fenv}[1]%
{\ensuremath{\,\overrightarrow{\operatorname{env}}_{#1}}}
\newcommand{\benv}[1]%
{\ensuremath{\,\overleftarrow{\operatorname{env}}_{#1}}}
\newcommand{\emp}{\ensuremath{\varnothing}}
\newcommand{\infconv}{\ensuremath{\mbox{\small$\,\square\,$}}}

\newcommand{\scal}[2]{\left\langle{#1},{#2}  \right\rangle}

\newcommand{\RR}{\ensuremath{\mathbb R}}

\newcommand{\RX}{\ensuremath{\,\left]-\infty,+\infty\right]}}

\newcommand{\dom}{\ensuremath{\operatorname{dom}}}

\newcommand{\gr}{\ensuremath{\operatorname{gr}}}

\newcommand{\reli}{\ensuremath{\operatorname{ri}}}
\newcommand{\inte}{\ensuremath{\operatorname{int}}}

\newcommand{\ran}{\ensuremath{\operatorname{ran}}}
\newcommand{\conv}{\ensuremath{\operatorname{conv}}}
\newcommand{\epi}{\ensuremath{\operatorname{epi}}}
\newcommand{\aff}{\ensuremath{\operatorname{aff}}}
\newcommand{\cconv}{\ensuremath{\overline{\operatorname{conv}}\,}}
\newcommand{\Fix}{\ensuremath{\operatorname{Fix}}}
\newcommand{\Id}{\ensuremath{\operatorname{Id}}}

\newcommand{\pinf}{\ensuremath{+\infty}}

{\begin{list}{}{%
\settowidth{\labelwidth}{\textrm{#1~}}%
\setlength{\leftmargin}{\labelwidth+\labelsep}}}
{\end{list}}
\newtheorem{theorem}{Theorem}[section]
\newtheorem{lemma}[theorem]{Lemma}
\newtheorem{corollary}[theorem]{Corollary}
\newtheorem{proposition}[theorem]{Proposition}
\newtheorem{definition}[theorem]{Definition}
\theoremstyle{plain}{\theorembodyfont{\rmfamily}
}
\theoremstyle{plain}{\theorembodyfont{\rmfamily}
}
\theoremstyle{plain}{\theorembodyfont{\rmfamily}
}
\theoremstyle{plain}{\theorembodyfont{\rmfamily}
\newtheorem{example}[theorem]{Example}}
\newtheorem{fact}[theorem]{Fact}
\theoremstyle{plain}{\theorembodyfont{\rmfamily}
\newtheorem{remark}[theorem]{Remark}}

\begin{document}

\title{\textrm{Near Equality, Near Convexity,\\ Sums of Maximally
Monotone Operators,\\ and
Averages of Firmly Nonexpansive Mappings}}

\author{
Heinz H.\ Bauschke\thanks{Mathematics, Irving K.\ Barber School,
University of British Columbia, Kelowna, British Columbia V1V 1V7, Canada. E-mail:
\texttt{heinz.bauschke@ubc.ca}.},
~Sarah M.\ Moffat\thanks{Mathematics, Irving K.\ Barber School, 
University of British Columbia, 
Kelowna, British Columbia V1V 1V7, Canada. E-mail:
\texttt{sarah.moffat@ubc.ca}.},
~and Xianfu\
Wang\thanks{Mathematics, Irving K.\ Barber School, 
University of British Columbia, 
Kelowna, British Columbia V1V 1V7, Canada.
E-mail:  \texttt{shawn.wang@ubc.ca}.}}

\date{April 29, 2011\\
\bigskip
\emph{\small Dedicated to Jonathan Borwein on the occasion of his 60th Birthday}}
\maketitle

\vskip 8mm

\begin{abstract} \noindent
We study nearly equal and nearly convex sets, 
ranges of maximally monotone operators,
and ranges and fixed points of convex combinations of firmly nonexpansive
mappings. The main result states that the range of an average of
firmly nonexpansive mappings is nearly equal to the average of
the ranges of the mappings. A striking application of this result
yields that 
the average of asymptotically regular
firmly nonexpansive mappings is also asymptotically regular.
Throughout, examples are provided to illustrate the theory.
We also obtain detailed information on the domain and range
of the resolvent average. 
\end{abstract}

{\small

}

\section{Overview}

Throughout, we assume that
\begin{equation}
\text{$X$ is
a real Euclidean space with inner
product $\scal{\cdot}{\cdot}$ and induced norm $\|\cdot\|$,}
\end{equation}
that
$m$ is a strictly positive integer, 
and that $I=\{1,\ldots,m\}$.

Our aim is to study range properties of sums of maximally monotone
operators as well as range and fixed point properties of firmly nonexpansive
mappings. The required notions of near convexity and near equality are
introduced in Section~\ref{sec:near}. Section~\ref{sec:mon} is concerned
with maximally monotone operators, while firmly nonexpansive mappings are
studied in Section~\ref{sec:firm}.
The notation we employ is standard and as in e.g.,
\cite{BC}, \cite{BorVan}, \cite{Rock70}, and \cite{Zali02} to which
we refer the reader for background material and further information.

\section{Near Equality and Near Convexity}
\label{sec:near}

In this section, we introduce near equality for sets and show that 
this notion is useful in the study of
nearly convex sets. 
These results are the key to study ranges of sums of maximal
monotone operators in sequel.
Let $C$ be a subset of $X$.
We use $\conv C$ and $\aff C$ for
 the \emph{convex hull} and \emph{affine hull}, respectively;
the \emph{closure} of $C$ is denoted by $\overline{C}$ and 
$\reli C$ 
is the \emph{relative interior} of a $C$
(i.e., 
the interior with respect to the affine hull of $C$).
See \cite[Chapter~6]{Rock70} for more on this fundamental notion.
The next result follows directly from the definition.

\begin{lemma}
\label{l:reli:cute}
Let $A$ and $B$ be subsets of $X$ such that
$A\subseteq B$ and $\aff A = \aff B$.
Then $\reli A \subseteq \reli B$.
\end{lemma}

\begin{fact}[Rockafellar] \label{f:ri}
Let $C$ and $D$ be convex subsets of $X$,
and let $\lambda \in \RR$.
Then the following hold.
\begin{enumerate}
\item
\label{f:ri:i--}
$\reli C$ and $\overline{C}$ are convex.
\item
\label{f:ri:i-}
$C\neq\varnothing$ $\Rightarrow$ $\reli C \neq\varnothing$.
\item
\label{f:ri:i}
$\overline{\reli C} = \overline{C}$.
\item
\label{f:ri:iii}
$\reli C = \reli\overline{C}$.
\item
\label{f:ri:ii}
$\aff\reli C = \aff C = \aff\overline{C}$.
\item
\label{f:ri:vi}
$\reli C = \reli D$
$\Leftrightarrow$
$\overline{C}= \overline{D}$
$\Leftrightarrow$
$\reli C \subseteq D \subseteq \overline{C}$.
\item \label{f:ri:scal} $\reli \lambda C=\lambda \reli C$.
\item \label{f:ri:sum}
$\reli(C+D)=\reli C+\reli D.$
\end{enumerate}
\end{fact}
\begin{proof}
\ref{f:ri:i--}\&\ref{f:ri:i-}: See \cite[Theorem~6.2]{Rock70}.
\ref{f:ri:i}\&\ref{f:ri:iii}:
See \cite[Theorem~6.3]{Rock70}.
\ref{f:ri:ii}: See \cite[Theorem~6.2]{Rock70}.
\ref{f:ri:vi}:
See \cite[Corollary~6.3.1]{Rock70}. 
\ref{f:ri:scal}:
See \cite[Corollary~6.6.1]{Rock70}.
\ref{f:ri:sum}:
See \cite[Corollary~6.6.2]{Rock70}.
\end{proof}

The key notion in this paper is defined next.

\begin{definition}[near equality]
\label{d:ne}
Let $A$ and $B$ be subsets of $X$.
We say that $A$ and $B$ are \emph{nearly
equal}, if
\begin{equation}
A\approx B \quad :\Leftrightarrow\quad
\text{$\overline{A} = \overline{B}$ ~~and~~ $\reli A = \reli B$.}
\end{equation}
\end{definition}

Observe that 
\begin{equation}
\label{e:mike}
A\approx B \;\Rightarrow\;
\inte A=\inte B
\end{equation}
since the relative interior coincides with the interior whenever
the interior is nonempty.

\begin{proposition}[equivalence relation]
The following hold for any subsets $A$, $B$, $C$ of $X$.
\begin{enumerate}
\item
$A \approx A$.
\item
$A \approx B$ $\Rightarrow$ $B \approx A$.
\item
$A \approx B$ and $B\approx C$
$\Rightarrow$ $A\approx C$.
\end{enumerate}
\end{proposition}

\begin{proposition}[squeeze theorem]
\label{p:squeeze}
Let $A,B,C$ be subsets of $X$ such that
$A\approx C$ and $A\subseteq B\subseteq C$. Then
$A\approx B\approx C$.
\end{proposition}
\begin{proof}
By assumption, $\overline{A}=\overline{C}$ and $\reli A=\reli C$.
Thus $\overline{A}=\overline{B} = \overline{C}$ and also 
$\aff(A)=\aff(\overline{A})=\aff(\overline{C})=\aff(C)$.
Hence 
$\aff A=\aff B=\aff C$ and so,
by Lemma~\ref{l:reli:cute}, 
$\reli A\subseteq \reli B\subseteq\reli C$.
Since $\reli A=\reli C$, we deduce that 
$\reli A=\reli B=\reli C$. 
Therefore, $A\approx B\approx C$.
\end{proof}

The equivalence relation ``$\approx$" is well suited 
for studying nearly convex sets, the definition of which we recall
next.

\begin{definition}[near convexity]
{\rm (See Rockafellar and Wets's \cite[Theorem~12.41]{Rock98}.)}
Let $A$ be a subset of $X$.
Then $A$ is \emph{nearly convex}
if there exists a convex subset $C$ of $X$ such that
$C \subseteq A \subseteq \overline{C}$.
\end{definition}

\begin{lemma}
\label{l:nearconset}
Let $A$ be a nearly convex subset of $X$, say
$C\subseteq A\subseteq \overline{C}$, where 
$C$ is a convex subset of $X$. 
Then
\begin{equation}
A\approx\overline{A}\approx\reli A\approx\conv A 
\approx\reli\conv A \approx C.
\end{equation}
In particular, the following hold. 
\begin{enumerate}
\item\label{o:c} $\overline{A}$ and $\reli A$ are convex.
\item\label{t:c} If $A\neq\varnothing$, then $\reli A\neq\varnothing$.
\end{enumerate}
\end{lemma}
\begin{proof} We have
\begin{equation}
C\subseteq A\subseteq \conv A\subseteq \overline{C}
\quad\text{and}\quad
C\subseteq A\subseteq \overline{A}\subseteq\overline{C}.
\end{equation}
Since $C\approx \overline{C}$ by Fact~\ref{f:ri}\ref{f:ri:iii}, 
it follows from Proposition~\ref{p:squeeze} that
\begin{equation}
A\approx \overline{A}\approx\conv A\approx C.
\end{equation}
This implies 
\begin{equation}
\reli(\reli A)=\reli (\reli C)=\reli C=\reli A
\end{equation}
and
\begin{equation}
\overline{\reli A}=\overline{\reli C}=\overline{C}=\overline{A}
\end{equation}
by Fact~\ref{f:ri}\ref{f:ri:i}. Therefore, $\reli A\approx A.$
Applying this to $\conv A$, which is nearly convex,
it also follows that $\reli\conv A \approx \conv A$. 
Finally, 
\ref{o:c} holds because $A\approx C$ while 
\ref{t:c} follows from $\reli A=\reli C$
and Fact~\ref{f:ri}\ref{f:ri:i-}.
\end{proof}

\begin{remark} The assumption of  
near convexity in Lemma~\ref{l:nearconset} is necessary:
consider $A=\mathbb{Q}$ when $X=\RR$.
\end{remark}

\begin{lemma}[characterization of near convexity]
\label{l:ruger}
Let $A\subseteq X$. 
Then the following are equivalent. 
\begin{enumerate}
\item 
\label{l:rugeri}
$A$ is nearly convex.
\item
\label{l:rugerii}
$A\approx\conv A$.
\item
\label{l:rugerii+}
$A$ is nearly equal to a convex set. 
\item
\label{l:rugerii++}
$A$ is nearly equal to a nearly convex set. 
\item
\label{l:rugeriii}
$\reli\conv A \subseteq A$.
\end{enumerate}
\end{lemma}
\begin{proof}
``\ref{l:rugeri}$\Rightarrow$\ref{l:rugerii}'': 
Apply Lemma~\ref{l:nearconset}. 
``\ref{l:rugerii}$\Rightarrow$\ref{l:rugeriii}'': 
Indeed, $\reli\conv A = \reli A \subseteq A$.
``\ref{l:rugeriii}$\Rightarrow$\ref{l:rugeri}'':
Set $C = \reli\conv A$.
By Fact~\ref{f:ri}\ref{f:ri:i},
$C \subseteq A \subseteq \cconv A
= \overline{\reli\conv A} = \overline{C}$.
``\ref{l:rugerii}$\Rightarrow$\ref{l:rugerii+}'': Clear.
``\ref{l:rugerii+}$\Rightarrow$\ref{l:rugeri}'': 
Suppose that $A\approx C$, where $C$ is convex. 
Then, using Fact~\ref{f:ri}\ref{f:ri:i}, 
$\reli C = \reli A \subseteq A \subseteq
\overline{A} = \overline{C} = \overline{\reli C}$. 
Hence $A$ is nearly convex. 
``\ref{l:rugerii+}$\Rightarrow$\ref{l:rugerii++}'': Clear. 
``\ref{l:rugerii++}$\Rightarrow$\ref{l:rugeri}'': 
Suppose $A\approx B$, where $B$ is nearly convex.
Then 
$\overline{A}=\overline B$ and $\reli A=\reli B$. 
As $B$ is nearly convex,
we obtain from Lemma~\ref{l:nearconset} 
and Fact~\ref{f:ri}\ref{f:ri:i} that
\begin{equation}
\reli \conv B=\reli B=\reli A\subseteq A\subseteq \overline{A}=
\overline{B}=\overline{\conv B}=
\overline{\reli \conv B}.
\end{equation}
Therefore, $A$ is nearly convex.
\end{proof}

\begin{remark}
The condition appearing in Lemma~\ref{l:ruger}\ref{l:rugeriii}
was also used by Minty \cite{Minty61} and named ``almost-convex''.
\end{remark}

\begin{remark}
\label{r:mike}
Br\'ezis and Haraux \cite{B-H} define, 
for two subsets $A$ and $B$ of $X$, 
\begin{equation}
A\simeq B \quad :\Leftrightarrow\quad
\text{$\overline{A} = \overline{B}$ ~~and~~ $\inte A = \inte B$.}
\end{equation}
\begin{enumerate}
\item
In view of \eqref{e:mike}, it is clear that 
$A\approx B$ $\Rightarrow$ $A\simeq B$. 
\item
On the other hand, 
$A\simeq B$ $\not\Rightarrow$ $A\approx B$:
indeed, consider $X=\mathbb{R}^2$,
$A = \mathbb{Q}\times\{0\}$, and 
$B = \mathbb{R}\times\{0\}$. 
\item 
The implications 
\ref{l:rugerii+}$\Rightarrow$\ref{l:rugeri}
and \ref{l:rugerii}$\Rightarrow$\ref{l:rugeri}
in Lemma~\ref{l:ruger} fails for $\simeq$: 
indeed, consider $X=\mathbb{R}^2$,
$A=(\mathbb{R}\smallsetminus\{0\})\times\{0\}$
and $C=\conv A = \mathbb{R}\times\{0\}$.
Then $C$ is convex and $A\simeq C$.
However, $A$ is not nearly convex 
because $\reli A \neq \reli\overline{A}$. 
\end{enumerate}
\end{remark}

\begin{proposition}
\label{p:char:near}
Let $A$ and $B$ be nearly convex subsets of $X$. 
Then the following are equivalent. 
\begin{enumerate}
\item\label{char:i} $A\approx B$.
\item\label{char:ii} $\overline{A}=\overline{B}$.
\item \label{char:iii} $\reli A=\reli B$.
\item \label{char:iv} $\overline{\conv A}=\overline{\conv B}$.
\item \label{char:v} $\reli\conv A=\reli \conv B$.
\end{enumerate}
\end{proposition}
\begin{proof}
``\ref{char:i}$\Rightarrow$\ref{char:ii}'': 
This is clear from the definition of $\approx$.
``\ref{char:ii}$\Rightarrow$\ref{char:iii}'': 
$\reli{\overline{A}}=\reli A$
and $\reli\overline{B}=\reli B$ by Lemma~\ref{l:nearconset}.
``\ref{char:iii}$\Rightarrow$\ref{char:iv}'': 
$\overline{\reli A}=\overline{\conv A}$
and $\overline{\reli B}=\overline{\conv B}$ by Lemma~\ref{l:nearconset}.
``\ref{char:iv}$\Rightarrow$\ref{char:v}'': 
$\reli{\overline{\conv A}}=\reli \conv A$
and $\reli{\overline{\conv B}}=\reli \conv B$.
``\ref{char:v}$\Rightarrow$\ref{char:i}'': 
Lemma~\ref{l:nearconset} gives that
$\reli \conv A =\reli A$ and $\reli \conv B =\reli B$ so that
$\reli A=\reli B$, 
$\overline{\reli \conv A}=\overline{\conv A}=\overline{A}$ and
$\overline{\reli \conv B}=\overline{\conv B}=\overline{B}$ 
so that $\overline{A}=\overline{B}$.
Hence \ref{char:i} holds.
\end{proof}

In order to study addition of nearly convex sets, 
we require the following result.

\begin{lemma}\label{t:split} 
Let $(A_i)_{i\in I}$ be a family of nearly convex subsets of $X$,
and let $(\lambda_i)_{i\in I}$ be a family of real numbers.
Then 
$\sum_{i\in I}\lambda_{i}A_{i}$ is nearly convex, and 
$\reli(\sum_{i\in I}\lambda_{i}A_{i})=\sum_{i\in I}\lambda_{i}\reli
A_{i}$.
\end{lemma}
\begin{proof}
For every $i\in I$, 
there exists a convex subset $C_{i}$ of $X$ such that
$C_{i}\subseteq A_{i}\subseteq \overline{C_{i}}$. 
We have
\begin{equation}
\sum_{i\in I}\lambda_{i}C_{i}\subseteq 
\sum_{i\in I}\lambda_{i}A_{i}
\subseteq \sum_{i\in I}\lambda_{i}\overline{C_{i}}\subseteq
\overline{\sum_{i\in I}\lambda_{i}C_{i}}, 
\end{equation}
which yields the near convexity of 
$\sum_{i\in I}\lambda_{i}A_{i}$ and
$\sum_{i\in I}\lambda_i A_i \approx \sum_{i\in I}\lambda_i C_i$
by Lemma~\ref{l:nearconset}. 
Moreover, by Fact~\ref{f:ri}\ref{f:ri:scal}\&\ref{f:ri:sum} 
and Lemma~\ref{l:nearconset}, 
\begin{equation}
\reli\Big(\sum_{i\in I}\lambda_{i}A_{i}\Big)
=\reli\Big(\sum_{i\in I}\lambda_{i}C_{i}\Big)
=\sum_{i\in I}\reli\big(\lambda_{i}C_{i}\big)=
\sum_{i\in I}\lambda_{i}\reli C_{i}
=\sum_{i\in I}\lambda_{i}\reli A_{i}.
\end{equation}
This completes the proof.
\end{proof}

\begin{theorem}
\label{t:ruger}
Let $(A_i)_{i\in I}$ be a family of nearly convex subsets of $X$,
and let $(B_i)_{i\in I}$ be a family of subsets of $X$ such that
$A_i \approx B_i$, for every $i\in I$. 
Then $\sum_{i\in I}A_{i}$ is nearly convex and
$ \sum_{i\in I}A_{i}\approx \sum_{i\in I}B_{i}$.
\end{theorem}
\begin{proof} 
Lemma~\ref{l:ruger} implies that $B_{i}$ is nearly convex,
for every $i\in I$. 
By Lemma~\ref{t:split}, 
we have that $\sum_{i\in I}A_{i}$ is nearly convex and
\begin{equation}
\reli \sum_{i\in I}A_{i}=\sum_{i\in I}\reli A_{i}=
\sum_{i\in I}\reli B_{i}=\reli\sum_{i\in I}B_{i}.
\end{equation}
Furthermore, 
\begin{equation}
\overline{\sum_{i\in I}A_{i}}=
\overline{\sum_{i\in I}\overline{A_{i}}}
=\overline{\sum_{i\in I}\overline{B_{i}}}
=\overline{\sum_{i\in I}B_{i}}
\end{equation}
and the result follows. 
\end{proof}

\begin{remark}
Theorem~\ref{t:ruger} fails without the near convexity assumption:
indeed, when $X=\mathbb{R}$ and $m=2$, 
consider  $A_1 = A_2 = \mathbb{Q}$ and
$B_1=B_2=\RR\smallsetminus\mathbb{Q}$.
Then $A_i\approx B_i$, for every $i\in I$, yet 
$A_1+A_2 = \mathbb{Q} \not\approx\RR = B_1+B_2$.
\end{remark}

\begin{theorem}\label{t:equivalent}
Let $(A_i)_{i\in I}$ be a family of nearly convex subsets of $X$,
and let $(\lambda_i)_{i\in I}$ be a family of real numbers.
For every $i\in I$, take $B_i\in \big\{ A_i,
\overline{A_{i}},\conv A_{i},\reli A_{i},\reli\conv A_i\big\}$.
Then
\begin{equation}\label{t:review}
\sum_{i\in I}\lambda_{i}A_{i}\approx \sum_{i\in I}\lambda_{i}B_{i}. 
\end{equation}
\end{theorem}
\begin{proof}
By Lemma~\ref{l:nearconset}, $A_{i}\approx B_{i}$ for every $i\in I$. 
Now apply Theorem~\ref{t:ruger}.
\end{proof}

\begin{corollary} 
\label{c:sprink}
Let $(A_i)_{i\in I}$ be a family of nearly convex subsets of $X$,
and let $(\lambda_i)_{i\in I}$ be a family of real numbers.
Suppose that there exists $j\in I$ such that $\lambda_j \neq 0$. 
Then
\begin{equation}
\label{t:interior}
\big(\inte\lambda_{j}A_{j}\big)+\sum_{i\in I\smallsetminus\{j\}}\lambda_{i}
\overline{A_{i}}\subseteq \inte \sum_{i\in I}\lambda_{i}A_{i};
\end{equation}
consequently, the following hold.
\begin{enumerate}
\item 
\label{sprink1}
If $(0\in \inte A_{j})\cap \bigcap_{i\in
I\smallsetminus\{j\}}\overline{A_i}$, then 
$0\in  \inte \sum_{i\in I}\lambda_{i}A_{i}$.
\item 
\label{sprink2}
If $A_{j}=X$, then $\sum_{i\in I}\lambda_{i}A_{i}=X$.
\end{enumerate}
\end{corollary}
\begin{proof}
By Theorem~\ref{t:equivalent},
$\reli (\lambda_{j}A_{j}+\sum_{i\in I\smallsetminus\{j\}}
\lambda_{i}\overline{A_{i}})=
\reli\sum_{i\in I}\lambda_{i}A_{i}$. 
Since
\begin{equation}
\big(\inte\lambda_{j}A_{j}\big)+
\sum_{i\in I\smallsetminus\{j\}}\lambda_{i}\overline{A_{i}}
\subseteq\reli \Big(\lambda_{j}A_{j}+
\sum_{i\in I\smallsetminus\{j\}}\lambda_{i}\overline{A_{i}}\Big),
\end{equation}
and $\big(\inte\lambda_{j}A_{j}\big)+\sum_{i\in I\smallsetminus\{j\}}
\lambda_{i}\overline{A_{i}}$ is an open set,
\eqref{t:interior} follows. 
\ref{sprink1} and \ref{sprink2} follow 
from \eqref{t:interior}.
\end{proof}

We develop a complementary cancellation result whose
proof relies on
R{\aa}dstr{\"o}m's cancellation.

\begin{fact}
\label{f:Rad}
{\rm (See \cite{Radstrom}.)}
Let $A$ be a nonempty subset of $X$,
let $E$ be a nonempty bounded subset of $X$, and
let $B$ be a nonempty closed convex subset of $X$
such that $A+E \subseteq B+E$.
Then $A\subseteq B$.
\end{fact}

\begin{theorem}
Let $A$ and $B$ be nonempty nearly convex subsets of $X$,
and let $E$ be a nonempty compact subset of $X$
such that $A+E\approx B+E$. Then $A\approx B$.
\end{theorem}
\begin{proof}
We have
$A+E \subseteq \overline{A+E}
= \overline{B+E} = \overline{B}+E$.
Fact~\ref{f:Rad} implies $A\subseteq \overline{B}$;
hence, $\overline{A}\subseteq\overline{B}$.
Analogously, $\overline{B}\subseteq\overline{A}$ and thus
$\overline{A}=\overline{B}$.
Now apply Proposition~\ref{p:char:near}. 
\end{proof}

Finally, we give a result concerning the interior of nearly convex sets.
\begin{proposition}
Let $A$ be a nearly convex subset of $X$. Then
$\inte A=\inte \conv A=\inte \overline{A}$.
\end{proposition}

\begin{proof}
By Lemma~\ref{l:nearconset},
$A\approx B$, where $B\in\big\{\overline{A},\conv A\big\}$.
Now recall \eqref{e:mike}. 
\end{proof}

\section{Maximally Monotone Operators}

\label{sec:mon}

Let $A\colon X\To X$, i.e., $A$ is a set-valued operator on $X$ in the
sense that $(\forall x\in X)$ $Ax\subseteq X$. The graph of $A$ is
denoted by $\gr A$. Then $A$ is \emph{monotone} (on $X$) if
\begin{equation}
(\forall (x,x^*)\in\gr A)(\forall (y,y^*)\in\gr A)\quad
\scal{x-y}{x^*-y^*} \geq 0,
\end{equation}
and $A$ is \emph{maximally monotone} if $A$ admits no proper monotone
extension. Classical examples of monotone operators are
subdifferential operators of functions that are convex, lower
semicontinuous, and proper;
linear operators with a positive symmetric part.
See, e.g., \cite{BC}, \cite{BorVan}, \cite{Brez73}, \cite{BurIus},
\cite{Rock98}, \cite{Simons98}, \cite{Simons},
\cite{Zali02}, \cite{Zeidler2a}, and \cite{Zeidler2b}
for applications and further information. 
As usual, the domain and range of $A$ are denoted by
 $\dom A=\{x\in X:\ Ax\neq\varnothing\}$ and
 $\ran A=\bigcup_{x\in X}Ax$ respectively; $\dom f
 = \menge{x\in X}{f(x)<\pinf}$ stands for the domain
of a function $f:X\to\RX$.

\begin{fact}[Rockafellar]
\label{f:sumrule}
{\rm (See \cite{Rock70c} or \cite[Theorem~12.44]{Rock98}.)}
Let $A$ and $B$ be maximally monotone on $X$.
Suppose that $\reli\dom A \cap \reli\dom B\neq\emp$.
Then $A+B$ is maximally monotone.
\end{fact}

\begin{fact}[Minty]
\label{f:Minty61}
{\rm (See \cite{Minty61} or \cite[Theorem~12.41]{Rock98}.)}
Let $A\colon X\To X$ be maximally monotone.
Then $\dom A$ and $\ran A$ are nearly convex.
\end{fact}

\begin{remark}
Fact~\ref{f:Minty61} is optimal in the sense that
the domain or the range of a maximally monotone operator
may fail to be convex---even for a subdifferential operator---see,
e.g., \cite[page~555]{Rock98}.
\end{remark}

Sometimes 
quite precise information is available
on the range of the sum of two maximally monotone operators.
To formulate the corresponding statements, we need to
review a few notions.

\begin{definition}[Fitzpatrick function]
{\rm (See \cite{Fitz}, and also \cite{BS} or \cite{MLT}.)}
Let $A\colon X\To X$. Then the \emph{Fitzpatrick function}
associated with $A$ is
\begin{equation}
F_A\colon X\times X\to\RX\colon
(x,x^*)\mapsto
\sup_{(a,a^*)\in\gr A}
\big(\scal{x}{a^*}+\scal{a}{x^*}-\scal{a}{a^*}\big).
\end{equation}
\end{definition}

\begin{example}[energy]
{\rm (See, e.g., \cite[Example~3.10]{BMS}.)}
\label{ex:energy}
Let $\Id: X\to X\colon x\mapsto x$ be the 
\emph{identity operator}. 
Then $F_{\Id}\colon X\times X \to \RR\colon (x,x^*)\mapsto
\tfrac{1}{4}\|x+x^*\|^2$.
\end{example}

\begin{definition}[Br\'ezis-Haraux]
{\rm (See \cite{B-H}.)}
Let $A\colon X\to X$ be monotone.
Then $A$ is \emph{\reckti}\ (which is also known as star-monotone or $3^*$
monotone),
if
\begin{equation}
\dom A \times \ran A \subseteq \dom F_A.
\end{equation}
\end{definition}

\begin{remark}
If $A\colon X\To X$ is maximally monotone and \reckti,
then one obtains the ``rectangle''
$\overline{\dom F_A} = \overline{\dom A} \times \overline{\ran A}$, which
prompted Simons \cite{Simons06} to call such an operator rectangular.
\end{remark}

\begin{fact}
\label{f:reckti}
Let $A$ and $B$ be monotone on $X$,
let $C\colon X\to X$ be linear and monotone, let $\alpha>0$,
and let $f\colon X\to\RX$ be convex, lower semicontinuous, and
proper.
Then the following hold.
\begin{enumerate}
\item
\label{f:recktiinv}
$A$ is \reckti\
$\Leftrightarrow$
$A^{-1}$ is \reckti.
\item
\label{f:recktigamma}
$A$ is \reckti\ $\Leftrightarrow$
$\alpha A$ is \reckti.
\item
\label{f:recktisubdiff}
$\partial f$ is maximally monotone and \reckti.
\item
\label{f:recktilin}
$C$ is \reckti\
$\Leftrightarrow$
$C^*$ is \reckti\
$\Leftrightarrow$
$(\exists \gamma>0)(\forall x\in X)$ $\scal{x}{Cx}\geq \gamma\|Cx\|^2$.
\item
\label{f:recktisum}
$(\dom A\cap \dom B)\times X \subseteq \dom F_B$
$\Rightarrow$
$A+B$ is \reckti.
\end{enumerate}
\end{fact}
\begin{proof}
\ref{f:recktiinv}\&\ref{f:recktigamma}:
This follows readily from the definitions.
\ref{f:recktisubdiff}: The fact that $\partial f$
is \reckti\ was pointed out in \cite[Example~1 on page~166]{B-H}.
For maximal monotonicity of $\partial f$,
see \cite{Moreau65} (or \cite[Corollary~31.5.2]{Rock70}
or \cite[Theorem~12.17]{Rock98}).
\ref{f:recktilin}: \cite[Proposition~2 and Remarque~2 on
page~169]{B-H}.
\ref{f:recktisum}: \cite[Proposition~24.17]{BC}.
\end{proof}

\begin{example}
{\rm (See also \cite[Example~3 on page~167]{B-H} or
\cite[Example~6.5.2(iii)]{AusTeb}.)}
\label{ex:resolvents}
Let $A\colon X\To X$ be maximally monotone.
Then $A+\Id$ and $(A+\Id)^{-1}$  are maximally monotone
and \reckti.
\end{example}
\begin{proof}
Combining Fact~\ref{f:reckti}\ref{f:recktisum}
and Example~\ref{ex:energy}, we see that
$A+\Id$ is \reckti.
Furthermore, $A+\Id$ is maximally monotone
by Fact~\ref{f:sumrule}.
Using Fact~\ref{f:reckti}\ref{f:recktiinv},
we see that $ (\Id+A)^{-1}$ is maximally monotone and \reckti.
\end{proof}

\begin{proposition}
\label{p:BMS}
{\rm (See \cite[Proposition~4.2]{BMS}.)}
Let $A$ and $B$ be monotone on $X$, and let $(x,x^*)\in X\times X$.
Then
$F_{A+B}(x,x^*) \leq \big(F_A(x,\cdot)\infconv F_B(x,\cdot)\big)(x^*)$.
\end{proposition}

\begin{lemma}
\label{l:claude}
Let $A$ and $B$ be \reckti\ on $X$.
Then $A+B$ is \reckti.
\end{lemma}
\begin{proof}
Clearly, $\dom(A+B) = (\dom A)\cap(\dom B)$,
and $\ran(A+B) \subseteq \ran A + \ran B$.
Take $x\in\dom(A+B)$ and $y^*\in\ran(A+B)$.
Then there exist $a^*\in\ran A$ and $b^*\in\ran B$
such that $a^*+b^*=y^*$.
Furthermore, $(x,a^*)\in(\dom A)\times(\ran A)\subseteq\dom F_A$ and
$(x,b^*)\in(\dom B)\times(\ran B)\subseteq\dom F_A$.
Using Proposition~\ref{p:BMS} and the
assumption that $A$ and $B$ are \reckti, we obtain
\begin{equation}
F_{A+B}(x,y^*) \leq F_A(x,a^*) + F_B(x,b^*) < \pinf.
\end{equation}
Therefore, $\dom(A+B)\times\ran(A+B)\subseteq\dom F_{A+B}$
and $A+B$ is \reckti.
\end{proof}

We are now ready to state the range result, which can be traced back to the
seminal paper by Br\'ezis and Haraux \cite{B-H}
(see also \cite{Simons98} or \cite{Simons06},
and \cite{Reich} for a Banach space version).
The useful finite-dimensional formulation we record here
was brought to light by Auslender and Teboulle~\cite{AusTeb}.

\begin{fact}[Br\'ezis-Haraux]
\label{f:B-H}
{\rm (See \cite[Theorem~6.5.1(b) and Theorem~6.5.2]{AusTeb}.)}
Let $A$ and $B$ be monotone on $X$ such that $A+B$ is maximally monotone.
Suppose that one of the following holds.
\begin{enumerate}
\item
$A$ and $B$ are \reckti.
\item
$\dom A\subseteq \dom B$ and $B$ is \reckti.
\end{enumerate}
Then
$\overline{\ran(A+B)}=\overline{\ran A+\ran B}$,
$\inte(\ran (A+B))=\inte(\ran A+\ran B)$, and 
$\reli\conv(\ran A + \ran B)\subseteq \ran(A+B)$. 
\end{fact}

Item~\ref{c:claudei} of the following result
also follows from Chu's \cite[Theorem~3.1]{Chu}.

\begin{theorem}
\label{c:claude}
Let $A$ and $B$ be monotone on $X$ such that
$A+B$ is maximally monotone. 
Suppose that one of the following holds.
\begin{enumerate}
\item
\label{c:claudei}
$A$ and $B$ are \reckti.
\item
\label{c:claudeii}
$\dom A\subseteq \dom B$ and $B$ is \reckti.
\end{enumerate}
Then $\ran(A+B)$ is nearly convex,
and $\ran(A+B)\approx \ran A+ \ran B$.
\end{theorem}
\begin{proof}
The near convexity of $\ran(A+B)$ follows from Fact~\ref{f:Minty61}.
Using Fact~\ref{f:B-H} and Fact~\ref{f:ri}\ref{f:ri:i},
\begin{subequations}
\begin{align}
\reli\conv(\ran A+ \ran B) &\subseteq \ran(A+B) \subseteq
\ran A + \ran B \subseteq \cconv(\ran A + \ran B)\\
&= \overline{\reli\conv(\ran A + \ran B)}.
\end{align}
\end{subequations}
Proposition~\ref{p:squeeze} and Lemma~\ref{l:nearconset} imply
$\ran(A+B)\approx\ran A +\ran B\approx \reli\conv(\ran A+ \ran B)$. 
\end{proof}

\begin{remark}
Considering $A+0$, where $A$ is the rotator by $\pi/2$ on $\RR^2$
which is not \reckti, we see that
$A+B$ need not be rectangular under assumption
\ref{c:claudeii} in Theorem~\ref{c:claude}.
\end{remark}

If we let $S_i=\ran A_i$ and $\lambda_i=1$ 
for every $i\in I$ in Theorem~\ref{c:punch}, 
then we obtain a result that is related to Pennanen's
\cite[Corollary~6]{Teemu}.

\begin{theorem}
\label{c:punch}
Let $(A_i)_{i\in I}$ be a family of maximally monotone
\reckti\ operators on $X$
with $\bigcap_{i\in I}\reli\dom A_i \neq\emp$, 
let $(S_i)_{i\in I}$ be a family of subsets of $X$ such that
\begin{equation}
(\forall i\in I)\quad
S_i\in\big\{\ran A_{i}, 
\overline{\ran A_{i}}, \reli (\ran A_{i}), \reli(\conv\ran
A_{i})\big\},
\end{equation}
and let $(\lambda_i)_{i\in I}$ 
be a family of strictly positive real numbers.
Then
$\sum_{i\in I}\lambda_iA_i$ is maximally monotone, 
\reckti, and 
$\ran\sum_{i\in I}\lambda_iA_i\approx
\sum_{i\in I}\lambda_iS_{i}$ is nearly convex. 
\end{theorem}
\begin{proof}
To see that $\sum_{i\in I}\lambda_iA_i$ is maximally monotone and
\reckti, use Fact~\ref{f:sumrule}, Lemma~\ref{l:claude}, and
induction. 
With Theorem~\ref{t:ruger}, 
Fact~\ref{f:sumrule} and Lemma~\ref{l:claude} in mind, 
Theorem~\ref{c:claude}\ref{c:claudei}
and induction yields
$\ran\sum_{i\in I}\lambda_iA_i\approx
\sum_{i\in I}\lambda_i\ran A_{i}$ and the near convexity. 
Finally, as $\ran A_{i}$ is nearly convex for every $i\in I$
by Fact~\ref{f:Minty61},
$\ran\sum_{i\in I}\lambda_iA_i\approx
\sum_{i\in I}\lambda_iS_{i}$ 
 follows from Theorem~\ref{t:equivalent}.
\end{proof}

The main result of this section is the following.

\begin{theorem}
\label{c:berlin}
Let $(A_i)_{i\in I}$ be a family of maximally monotone
\reckti\ operators on $X$ 
such that $\bigcap_{i\in I}\reli\dom A_i \neq\emp$, 
let $(\lambda_i)_{i\in I}$ be a family of strictly positive real
numbers, and let $j\in I$. 
Set
\begin{equation}
A = \sum_{i\in I} \lambda_i A_i.
\end{equation}
Then the following hold.
\begin{enumerate}
\item
\label{c:berlini}
If $\sum_{i\in I}\lambda_i\ran A_i = X$, then $\ran A = X$.
\item
\label{c:berlini+}
If $A_j$ is surjective, then $A$ is surjective.
\item
\label{c:berlinlong}
If $0\in \bigcap_{i\in I}\overline{\ran A_{i}}$, then
$0\in\overline{\ran A}$.
\item
\label{c:berlinii}
If 
$0\in(\inte\ran A_j)\cap \bigcap_{i\in I\smallsetminus\{j\}}
\overline{\ran A_i}$, then $0\in\inte\ran A$.
\end{enumerate}
\end{theorem}
\begin{proof}
Theorem~\ref{c:punch} implies that
$\ran\sum_{i\in I}\lambda_iA_i\approx
\sum_{i\in I}\lambda_i\ran A_i$ is nearly convex.
Hence
\begin{equation}\label{c:one}
\reli\ran A = \reli \ran\sum_{i\in I}\lambda_iA_i =
\reli\Big(\sum_{i\in I}\lambda_i\ran A_i\Big) =
\sum_{i\in I}\lambda_i\reli\ran A_i
\end{equation}
and 
\begin{equation}\label{c:two}
\overline{\ran A} = \overline{\ran\sum_{i\in I}\lambda_iA_i }=
\overline{\sum_{i\in I}\lambda_i\ran A_i}\,. 
\end{equation}
\ref{c:berlini}:
Indeed, using \eqref{c:one}, 
$X=\reli X = \reli \sum_{i\in I}\lambda_i\ran A_i
=\reli\ran A \subseteq \ran A\subseteq X$. 
\ref{c:berlini+}:
Clear from \ref{c:berlini}.
\ref{c:berlinlong}: 
It follows from \eqref{c:two} that
$0 \in  \sum_{i\in I}\lambda_i\overline{\ran A_i} \subseteq
\overline{\sum_{i\in I}\lambda_i{\ran A_i}} = \overline{\ran A}$.
\ref{c:berlinii}:
By Fact~\ref{f:Minty61}, $\ran A_i$ is nearly convex for every 
$i\in I$. 
Thus,
$0\in \inte\sum_{i\in I}\lambda_i \ran A_i$ by 
Corollary~\ref{c:sprink}\ref{sprink1}. 
On the other hand, \eqref{c:one} implies that 
$\inte \sum_{i\in I}\lambda_i \ran A_i\subseteq
\reli\sum_{i\in I}\lambda_i \ran A_i = \reli\ran A$.
Altogether, $0\in\inte\ran A$. 
\end{proof}

\section{Firmly Nonexpansive Mappings}

\label{sec:firm}

To find zeros of maximally monotone operators, 
one often utilizes firmly
nonexpansive mappings \cite{BC,Comb94,EckBer,rockprox}. In this section,
we apply the result of Section~\ref{sec:mon} to firmly nonexpansive
mappings. 
Let $T\colon X\to X$.
Recall that $T$ is \emph{firmly nonexpansive}
(see also Zarantonello's seminal work \cite{Zara} for further results)
if 
\begin{equation}
(\forall x\in X)(\forall y\in X)\quad
\scal{x-y}{Tx-Ty}\geq \|Tx-Ty\|^2.
\end{equation}

The following characterizations are well known.

\begin{fact}
\label{f:Goebel}
{\rm (See, e.g., \cite{BC}, \cite{GK}, or \cite{GR}.)}
Let $T\colon X\to X$. Then the following are equivalent.
\begin{enumerate}
\item $T$ is firmly nonexpansive.
\item $(\forall x\in X)(\forall y\in X)$
$\|Tx-Ty\|^2 + \|(\Id-T)x-(\Id-T)y\|^2 \leq \|x-y\|^2$.
\item $(\forall x\in X)(\forall y\in X)$
$0\leq \scal{Tx-Ty}{(\Id-T)x-(\Id-T)y}$.
\item $\Id-T$ is firmly nonexpansive.
\item $2T-\Id$ is \emph{nonexpansive}, i.e.,
Lipschitz continuous with constant $1$.
\end{enumerate}
\end{fact}

Minty \cite{Minty62} first observed---while Eckstein and Bertsekas
\cite{EckBer}
made this fully precise---a fundamental
correspondence between maximally monotone operators
and firmly nonexpansive mappings. It is based on the 
\emph{resolvent} of $A$, 
\begin{equation}
J_A := (\Id+A)^{-1},
\end{equation}
which satisfies the useful identity
\begin{equation}
\label{e:nice}
J_{A} + J_{A^{-1}} = \Id,
\end{equation}
and which allows for the beautiful \emph{Minty parametrization} 
\begin{equation}
\label{e:Mintypar}
\gr A = 
\menge{(J_Ax,x-J_Ax)}{x\in X}
\end{equation}
of the graph of $A$.

\begin{fact}
\label{f:Minty}
{\rm (See \cite{EckBer} and \cite{Minty62}.)}
Let $T\colon X\to X$ and let $A\colon X\To X$.
Then the following hold.
\begin{enumerate}
\item
\label{f:Mintyi}
If $T$ is firmly nonexpansive,
then $B := T^{-1}-\Id$ is maximally monotone and $J_B=T$.
\item
If $A$ is maximally monotone, then $J_A$ has full domain, and it is
single-valued and firmly nonexpansive.
\end{enumerate}
\end{fact}

\begin{corollary}
\label{c:pubs}
Let $T\colon X\to X$ be firmly nonexpansive.
Then $T$ is maximally monotone and \reckti,
and $\ran T$ is nearly convex.
\end{corollary}
\begin{proof}
Combine Example~\ref{ex:resolvents},
Fact~\ref{f:Minty}\ref{f:Mintyi}, and Fact~\ref{f:Minty61}.
\end{proof}

It is also known that the class of firmly nonexpansive
mappings is closed under taking convex combinations.
For completeness, we include a short proof of this result.

\begin{lemma}
\label{l:punch}
Let $(T_i)_{i\in I}$ be a family of firmly nonexpansive
mappings on $X$,
and let $(\lambda_i)_{i\in I}$ be a family of strictly positive
real numbers such that $\sum_{i\in I}\lambda_i=1$.
Then $\sum_{i\in I}\lambda_iT_i$ is also firmly nonexpansive.
\end{lemma}
\begin{proof}
Set $T = \sum_{i\in I}\lambda_iT_i$.
By Fact~\ref{f:Goebel}, $2T_{i}-\Id$ is nonexpansive for every
$i\in I$,
so $2T-\Id=\sum_{i\in I}\lambda_{i}(2T_{i}-\Id)$ is also nonexpansive.
Applying Fact~\ref{f:Goebel} once more, we deduce that $T$ is firmly
nonexpansive.
\end{proof}

We are now ready for the first main result of this section.

\begin{theorem}[averages of firmly nonexpansive mappings]
\label{t:punch}
Let $(T_i)_{i\in I}$ be a family of firmly nonexpansive
mappings on $X$,
let $(\lambda_i)_{i\in I}$ be a family of strictly positive
real numbers such that $\sum_{i\in I}\lambda_i=1$, and let $j\in I$. 
Set $T = \sum_{i\in I} \lambda_iT_i$.
Then the following hold.
\begin{enumerate}
\item
\label{t:punchi}
$T$ is firmly nonexpansive
and $\ran T \approx \sum_{i\in I}\lambda_i\ran T_i$
is nearly convex.
\item
\label{t:punchii}
If $T_j$ is surjective,
then $T$ is surjective.
\item
\label{t:punchlong}
If $0\in \bigcap_{i\in I} \overline{\ran T_{i}}$, then
$0\in\overline{\ran T}$.
\item
\label{t:punchiii}
If
$0\in(\inte\ran T_j)\cap \bigcap_{i\in I\smallsetminus\{j\}}
\overline{\ran T_i}$,
then
$0\in\inte\ran T$.
\end{enumerate}
\end{theorem}
\begin{proof} By Corollary~\ref{c:pubs}, each $T_{i}$ is 
maximally monotone, rectangular and
$\ran T_{i}$ is nearly convex.
\ref{t:punchi}: Lemma~\ref{l:nearconset},
Lemma~\ref{l:punch}, and Theorem~\ref{c:punch}.
\ref{t:punchii}: Theorem~\ref{c:berlin}\ref{c:berlini+}. \ref{t:punchlong}:
Theorem~\ref{c:berlin}\ref{c:berlinlong}.
\ref{t:punchiii}: Theorem~\ref{c:berlin}\ref{c:berlinii}.
\end{proof}

The following averaged-projection operator plays
a role in methods for solving (potentially inconsistent) convex
feasibility problems because its fixed point set consists
of least-squares solutions; see, e.g.,
\cite[Section~6]{BBJAT}, \cite{BWW} and \cite{Comb94}
for further information.

\begin{example}
\label{ex:average}
Let $(C_i)_{i\in I}$ be a family of nonempty closed convex subsets of
$X$ with associated projection operators $P_i$,
and let $(\lambda_i)_{i\in I}$ be a family of strictly positive
real numbers such that $\sum_{i\in I}\lambda_i =1 $.
Then
\begin{equation}
\ran\sum_{i\in I}\lambda_iP_i \approx
\sum_{i\in I}\lambda_i C_i.
\end{equation}
\end{example}
\begin{proof}
This follows from Theorem~\ref{t:punch}\ref{t:punchi}
since $(\forall i\in I)$
$\ran P_i = C_i$.
\end{proof}

\begin{remark}
Let $C_1$ and $C_2$ be nonempty closed convex subsets of $X$
with associated projection operators $P_1$ and $P_2$ respectively,
and---instead of averaging as in Example~\ref{ex:average}---consider the
composition $T=P_2\circ P_1$, which is still \emph{nonexpansive}.
It is obvious that $\ran T\subseteq \ran P_2 = C_2$,
but $\ran T$ need not be even nearly convex:
indeed, suppose that $X=\RR^2$, let $C_2$ be the unit ball
centered at $0$ of radius $1$, and let $C_1 = \RR\times\{2\}$.
Then $\ran T$ is the intersection of the open upper halfplane
and the boundary of $C_2$, which is very far from being nearly convex.
Thus the near convexity part of Corollary~\ref{c:pubs} has no
counterpart for nonexpansive mappings.
\end{remark}

\begin{definition}
\label{d:fix}
Let $T\colon X\to X$ be firmly nonexpansive.
The set of \emph{fixed points} is denoted by
\begin{equation}
\Fix T = \menge{x\in X}{x=Tx}.
\end{equation}
We say that $T$ is \emph{asymptotically regular} if there
exists a sequence $(x_n)_\nnn$ in $X$ such that
$x_n-Tx_n\to 0$; equivalently, if
$0\in\overline{\ran(\Id-T)}$. 
\end{definition}

\begin{remark}
\label{r:fix}
If the sequence $(x_n)_\nnn$ in Definition~\ref{d:fix}
has a cluster point, say $\bar{x}$, then continuity of $T$
implies that $\bar{x}\in\Fix T$.
\end{remark}

The next result is a consequence of
fundamental work \cite{BBR} by Baillon, Bruck and Reich.

\begin{theorem}
Let $T\colon X\to X$ be firmly nonexpansive.
Then $T$ is asymptotically regular if and only if
for every $x_0\in X$, the sequence defined by
\begin{equation}
(\forall\nnn)\quad x_{n+1} = Tx_n
\end{equation}
satisfies $x_n-x_{n+1}\to 0$.
Moreover, if $\Fix T\neq\emp$, then $(x_n)_\nnn$ converges
to a fixed point; otherwise, $\|x_n\|\to\pinf$.
\end{theorem}
\begin{proof}
This follows from
\cite[Corollary~2.3, Theorem~1.2, and Corollary~2.2]{BBR}.
\end{proof}

Here is the second main result of this section.

\begin{theorem}[asymptotic regularity of the average]
\label{t:kool}
Let $(T_i)_{i\in I}$ be a family of firmly nonexpansive
mappings on $X$,
and let $(\lambda_i)_{i\in I}$ be a family of strictly positive
real numbers such that $\sum_{i\in I}\lambda_i=1$.
Suppose that $T_i$ is asymptotically regular, for every $i\in I$.
Then $\sum_{i\in I} \lambda_iT_i$ is also asymptotically regular.
\end{theorem}
\begin{proof}
Set $T = \sum_{i\in I} \lambda_iT_i$. Then
$$\Id-T=\sum_{i\in I}\lambda_{i}(\Id- T_{i}).$$
Since each $\Id-T_{i}$ is firmly nonexpansive and $0\in\overline{\ran (\Id-T_{i})}$ by the
asymptotic regularity of $T_{i}$,
the conclusion follows from 
Theorem~\ref{t:punch}\ref{t:punchlong}.
\end{proof}

\begin{remark}
\label{r:kool}
Consider Theorem~\ref{t:kool}.
Even when $\Fix T_i\neq\emp$, for every $i\in I$, it is impossible
to improve the conclusion to $\Fix\sum_{i\in I}\lambda_iT_i\neq\emp$.
Indeed, suppose that $X=\RR^2$,
and set $C_1 = \RR\times\{0\}$ and $C_2=\epi\exp$.
Set $T = \thalb P_{C_1} + \thalb P_{C_2}$.
Then $\Fix T_1= C_1$ and $\Fix T_2 = C_2$, yet
$\Fix T = \emp$.
\end{remark}

The proof of the following useful result is straightforward and hence
omitted.

\begin{lemma}
\label{l:last}
Let $A\colon X\To X$ be maximally monotone.
Then $J_A$ is asymptotically regular if and only if $0\in
\overline{\ran A}$. 
\end{lemma}

We conclude this paper 
with an application to the resolvent average of monotone operators.
Let $(A_i)_{i\in I}$ be a family of maximally monotone operators on
$X$. 
Compute and average the corresponding resolvents
to obtain $T := \sum_{i\in I}\lambda_i J_{A_i}$.
By Lemma~\ref{l:punch}, $T$ is firmly nonexpansive;
hence, again by Fact~\ref{f:Minty}, $T=J_A$ for some
maximally monotone operator $A$.
The operator $A$ is called the
\emph{resolvent average} of the family $(A_i)_{i\in I}$
with respect to the weights $(\lambda_i)_{i\in I}$;
it was analyzed in detail for
real symmetric positive semidefinite matrices
in \cite{BMWLAA}.

\begin{corollary}[resolvent average] \label{c:resaverage}
Let $(A_i)_{i\in I}$ be a family of maximally monotone operators on
$X$, let $j\in I$, and set 
\begin{equation}
A=\bigg(\sum_{i\in I}\lambda_{i}(\Id+A_{i})^{-1}\bigg)^{-1}-\Id.
\end{equation}
Then the following hold. 
\begin{enumerate}
\item\label{c:wi} $A$ is maximally monotone.
\item\label{c:wii} $\dom A\approx \sum_{i\in I}\lambda_{i}\dom A_{i}$.
\item\label{c:wiii} $\ran A\approx \sum_{i\in I}\lambda_{i}\ran A_{i}$.
\item\label{c:wiv} If $0\in\bigcap_{i\in I}\overline{\ran A_{i}}$, 
then $0\in\overline{\ran A}$.
\item\label{c:wv} 
If $0\in\inte\ran A_{j}\cap 
\bigcap_{i\in I\setminus\{j\}}\overline{\ran A_{i}}$,
then $0\in\inte\ran A$.
\item\label{c:wvi} If $\dom A_{j}=X$, then $\dom A=X$.
\item \label{c:wvii} If $\ran A_{j}=X$, then $\ran A=X$.
\end{enumerate}
\end{corollary}
\begin{proof}
Observe that
\begin{equation}
\label{con:resolvent}
J_{A}=\sum_{i\in I}\lambda_{i} J_{A_{i}}
\end{equation}
and 
\begin{equation}\label{inver:resolvent}
\quad J_{A^{-1}}=\sum_{i\in I}\lambda_{i} J_{A_{i}^{-1}}
\end{equation}
by using \eqref{e:nice}.
Furthermore, using \eqref{e:Mintypar}, we note that 
\begin{equation}
\ran J_{A}=\dom A \quad\text{and}\quad 
\ran J_{A^{-1}}=\ran A.
\end{equation}
\ref{c:wi}: This follows from \eqref{con:resolvent} and Fact~\ref{f:Minty}.
\ref{c:wii}: Apply Theorem~\ref{t:punch}\ref{t:punchi} to
$(J_{A_i})_{i\in I}$.
\ref{c:wiii}: Apply Theorem~\ref{t:punch}\ref{t:punchi} to 
$(\Id-J_{A_i})_{i\in I}$.
\ref{c:wiv}:  Combine Theorem~\ref{t:kool} and Lemma~\ref{l:last}. 
\ref{c:wv}:  Apply Theorem~\ref{t:punch}\ref{t:punchiii} 
to \eqref{inver:resolvent}.
\ref{c:wvi} and \ref{c:wvii}: These follow 
from  \ref{c:wii} and \ref{c:wiii}, respectively.
\end{proof}

\begin{remark}[proximal average]
In Corollary~\ref{c:resaverage},
one may also start from a family $(f_i)_{i\in I}$
of functions on $X$ that are convex, lower semicontinuous,
and proper, and with corresponding subdifferential operators
$(A_i)_{i\in I} = (\partial f_i)_{i\in I}$.
This relates to the \emph{proximal average}, 
$p$, of the family $(f_i)_{i\in I}$, where 
$\partial p$ is the resolvent average of
the family $(\partial f_i)_{i\in I}$.
See \cite{BGLW} for further information and references.
Corollary~\ref{c:resaverage}\ref{c:wvii} essentially states
that $p$ is \emph{supercoercive}
provided that some $f_j$ is.
Analogously, Corollary~\ref{c:resaverage}\ref{c:wv} shows that 
that \emph{coercivity} of $p$ follows from the coercivity of
some function $f_j$. 
Similar comments apply to \emph{sharp minima}; see
\cite[Lemma~3.1 and Theorem~4.3]{GHW} for details.
\end{remark}

\section*{Acknowledgments}
Heinz Bauschke was partially supported by the Natural Sciences and
Engineering Research Council of Canada and by the Canada Research Chair
Program.
Sarah Moffat was partially
supported by the Natural Sciences and Engineering Research Council
of Canada.
Xianfu Wang was partially
supported by the Natural Sciences and Engineering Research Council
of Canada.


\end{document}